\documentclass[11pt]{article}
\usepackage{mathtools,amsfonts,amsmath,amssymb,amsthm,cite,hyperref}
\hypersetup{colorlinks=true,linkcolor=blue,citecolor=blue,urlcolor=blue}

\newcommand{\R}{\mathbb{R}}
\newcommand{\Z}{\mathbb{Z}}
\newcommand{\dimH}{\dim_{\mathrm H}}
\newcommand{\leT}{\leq_{\mathrm T}}

\newtheorem{thm}{Theorem}[section]
\newtheorem{defn}[thm]{Definition}
\newtheorem{lem}[thm]{Lemma}
\newtheorem{cor}[thm]{Corollary}
\numberwithin{equation}{section}

\begin{document}

\title{A Zero--One Law for Subrings of the Reals\\
Closed under Binary Interleaving}
\author{Yi Wang\thanks{Qiuzhen College, Tsinghua University, Beijing, China.}}
\date{}
\maketitle

\begin{abstract}
Erd\H{o}s Problem 1154 asks whether every number in $[0,1]$ occurs as the
Hausdorff dimension of a subring or subfield of $\R$. We prove that a
subring of $\R$ closed under binary interleaving has Hausdorff dimension
either zero or one.
Consequently, every Turing ideal and its associated subfield
satisfy the same dichotomy.
\end{abstract}

\section{Introduction}\label{sec:intro}

Erd\H{o}s asked whether the Hausdorff dimensions of subrings and subfields
of $\R$ can take arbitrary values in $[0,1]$ \cite{Erdos1979}.

Erd\H{o}s and Volkmann proved that every $\alpha\in[0,1]$ is the Hausdorff
dimension of an additive subgroup of $\R$ \cite{ErdosVolkmann1966}.
Edgar and Miller proved that an analytic subring of $\R$ either has
Hausdorff dimension zero or is all of $\R$ \cite{EdgarMiller2003}.
Assuming the Continuum Hypothesis, Mauldin constructed subfields of every prescribed
Hausdorff dimension in $[0,1]$ \cite{Mauldin2016}.

In this paper we obtain a zero--one law under an additional closure condition. Given any binary expansions
\[
x=0.x_0x_1x_2\cdots,\qquad y=0.y_0y_1y_2\cdots,
\]
of $x,y\in[0,1)$, their \emph{binary interleaving} is the real number
\[
z=0.x_0y_0x_1y_1x_2y_2\cdots.
\]
Binary expansions need not be unique. The value of $z$ may depend on the chosen expansions.
We say that an additive subgroup $G\subseteq\R$ is \emph{closed under binary interleaving}
if, for every $x,y\in G\cap[0,1)$ and every choice of their binary
expansions, the resulting interleaving $z$ belongs to $G$.

\begin{thm}\label{thm:main}
Let $S\subseteq\R$ be a subring closed under binary interleaving. Then
\[
\dimH S\in\{0,1\}.
\]
\end{thm}

Subrings in Theorem~\ref{thm:main} are not required to contain $1$.

For an oracle $A\in2^\omega$, let $\R_A$ denote the field of all $A$-computable real numbers.

\begin{cor}\label{cor:turing}
Let $\mathcal I\subseteq2^\omega$ be a Turing ideal, and put
\[
F_{\mathcal I}:=\bigcup_{A\in\mathcal I}\R_A.
\]
Then $F_{\mathcal I}$ is a subfield of $\R$, and
\[
\dimH\mathcal I=\dimH F_{\mathcal I}\in\{0,1\}.
\]
\end{cor}

Corollary~\ref{cor:models} also answers the question posed in
\cite{MOInnerModels}: the reals of a transitive inner model have Hausdorff
dimension zero or one when computed in an outer model.

The proof first uses binary interleaving to obtain the identity
$\dimH(S\times S)=2\dimH S$, and then applies the Furstenberg-set theorem
of Orponen and Shmerkin \cite{OrponenShmerkin2023} to rule out intermediate
Hausdorff dimensions.

\section{Preliminaries}\label{sec:preliminaries}

We use the usual Euclidean metrics on subsets of $\R^k$. Replacing these
by maximum metrics does not change Hausdorff dimension. We also use
countable stability and invariance under bi-Lipschitz maps.

Equip $2^\omega$ with the metric
\[
\rho(X,Y)=
\begin{cases}
0,&X=Y,\\
2^{-\min\{n:X(n)\ne Y(n)\}},&X\ne Y.
\end{cases}
\]
Define the binary evaluation map by
\[
\beta(X)=\sum_{n=0}^{\infty}X(n)2^{-(n+1)}.
\]

The following dimension-preservation property of radix expansions is
recorded in \cite[p.~365, Remark]{MerzenichStaiger1994}.

\begin{lem}\label{lem:binary}
For $B\subseteq2^\omega$ and $C\subseteq(2^\omega)^2$,
\[
\dimH\beta[B]=\dimH B,
\qquad
\dimH(\beta\times\beta)[C]=\dimH C,
\]
where $(2^\omega)^2$ carries the maximum product metric
\[
D((X_0,X_1),(Y_0,Y_1))
=\max\{\rho(X_0,Y_0),\rho(X_1,Y_1)\}.
\]
\end{lem}

\begin{proof}
We give a dyadic covering proof for arbitrary subsets.
Both $\beta$ and $\beta\times\beta$ are Lipschitz, so it remains to prove
the reverse inequalities.

For $\sigma\in2^n$, let $[\sigma]=\{X\in2^\omega:X\upharpoonright n=\sigma\}$.
Then $\operatorname{diam}_\rho[\sigma]=2^{-n}$, and
$I_\sigma:=\beta([\sigma])=[k2^{-n},(k+1)2^{-n}]$ for the corresponding
integer $0\le k<2^n$. Let $I\subseteq[0,1]$ be an interval of diameter
$0<\delta<1$, and choose $n\ge0$ such that
$\delta\le2^{-n}<2\delta$. At most three of the closed intervals
$I_\sigma$, $\sigma\in2^n$, meet $I$. Writing
$W_I=\{\sigma\in2^n:I_\sigma\cap I\ne\varnothing\}$, we therefore have
\[
\beta^{-1}(I)\subseteq\bigcup_{\sigma\in W_I}[\sigma],
\qquad |W_I|\le3,
\qquad \operatorname{diam}_\rho[\sigma]<2\delta.
\]
Taking all intervals that meet $I$ includes both binary expansions at
dyadic endpoints.
Thus an interval cover $\{I_i\}$ of $\beta[B]$, with
$0<\delta_i:=\operatorname{diam}I_i\le\delta$, pulls back to a cover of
$B$ by cylinders of diameter less than $2\delta_i$ and total
$s$-cost at most $3\,2^s\sum_i\delta_i^s$.
Taking infima over covers and letting $\delta\downarrow0$ gives
\[
\mathcal H^s_\rho(B)\le3\,2^s\mathcal H^s(\beta[B])
\qquad(s>0),
\]
and hence $\dimH B\le\dimH\beta[B]$.

For the planar statement, work with the maximum metric on $[0,1]^2$.
A set of diameter $0<\delta<1$ is contained in a product of two intervals
of length $\delta$. Applying the preceding construction to both
coordinates, with the same $n$, covers its inverse image under
$\beta\times\beta$ by at most nine product cylinders of $D$-diameter
less than $2\delta$. Pulling back covers as above proves
$\dimH C\le\dimH(\beta\times\beta)[C]$.
\end{proof}

For $X,Y\in2^\omega$, write $X\leT Y$ if $X$ is Turing computable from
$Y$. The \emph{Turing join} is the interleaving map
\[
(X\oplus Y)(2n)=X(n),\qquad
(X\oplus Y)(2n+1)=Y(n).
\]
We use the standard computability terminology of \cite{Nies2009}.
For any binary codes $X,Y$ of $x,y\in[0,1)$, the real represented by
their interleaving is $\beta(X\oplus Y)$.

\begin{defn}\label{def:turing}
A nonempty set $\mathcal I\subseteq2^\omega$ is a \emph{Turing ideal} if
\begin{align*}
X\in\mathcal I\text{ and }Y\leT X&\Longrightarrow Y\in\mathcal I,\\
X,Y\in\mathcal I&\Longrightarrow X\oplus Y\in\mathcal I.
\end{align*}
\end{defn}

The following dimension identity for binary interleaving is a special
case of \cite[Theorem~21]{FernandezMartinezEtAl2019}.
To apply that result, use inverse interleaving to identify cylinders
in $2^\omega$ determined by $2n$-bit prefixes with products of two
cylinders determined by $n$-bit prefixes.
Their diameters are $2^{-2n}$ and $2^{-n}$, respectively. Both partitions
split into four at each step and satisfy the bounded-intersection
condition with constant $1$. Thus the theorem applies with $d=2$ and
$c=1$. We include a direct metric proof.

\begin{lem}\label{lem:join}
Let $J:(2^\omega)^2\to2^\omega$ be given by $J(X,Y)=X\oplus Y$.
For every $C\subseteq(2^\omega)^2$,
\[
\dimH J[C]=\frac12\dimH C.
\]
\end{lem}

\begin{proof}
Let $D$ be the maximum product metric of Lemma~\ref{lem:binary}.
If $u\ne v$ and $D(u,v)=2^{-n}$, then $Ju$ and $Jv$ first differ at
position $2n$ or $2n+1$. Consequently,
\[
\frac12D(u,v)^2\le\rho(Ju,Jv)\le D(u,v)^2.
\]
Since $D(u,w)\le\max\{D(u,v),D(v,w)\}$, its square $D^2$ is also a
metric. The displayed inequalities show that
$J:(C,D^2)\to(J[C],\rho)$ is bi-Lipschitz.

Lipschitz maps do not increase Hausdorff dimension, so applying this
fact to the displayed bi-Lipschitz map and its inverse gives
$\dimH J[C]=\dimH(C,D^2)$.
Moreover, $\operatorname{diam}_{D^2}U=(\operatorname{diam}_D U)^2$, so
\[
\mathcal H^s_{D^2}(C)=\mathcal H^{2s}_D(C)
\qquad(s>0).
\]
Thus squaring the metric $D$ halves Hausdorff dimension. Combining this
with the bi-Lipschitz equivalence above gives
\[
\dimH J[C]=\dimH(C,D^2)=\frac12\dimH C.
\]
\end{proof}

For arbitrary $A\subseteq\R^m$ and $B\subseteq\R^n$, Marstrand's product
theorem \cite{Marstrand1954} gives
\begin{equation}\label{eq:product-lower}
\dimH(A\times B)\ge\dimH A+\dimH B.
\end{equation}

Families of nonvertical lines are parametrized by point--line duality,
\[
\mathbf D(a,b)=\ell_{a,b}
:=\{(x,y)\in\R^2:y=ax+b\}.
\]
The Hausdorff dimension of $\mathbf D(P)$ is that of its parameter set
$P\subseteq\R^2$.

\begin{thm}[Orponen--Shmerkin {\cite[Theorem~1.1]{OrponenShmerkin2023}}]
\label{thm:OS}
Let $0<s<1$ and $s<t\le2$. There exists $\varepsilon(s,t)>0$ such that
if $K\subseteq\R^2$ and a family $\mathcal L$ of nonvertical affine
lines satisfy $\dimH\mathcal L\ge t$ and
$\dimH(K\cap\ell)\ge s$ for every $\ell\in\mathcal L$, then
\[
\dimH K\ge2s+\varepsilon(s,t).
\]
\end{thm}

Ren and Wang \cite[Theorem~1.1]{RenWang2023} subsequently obtained the
sharp bound $\dimH K\ge\min\{s+t,(3s+t)/2,s+1\}$; the earlier bound in
Theorem~\ref{thm:OS} suffices for our proof.

\section{Proofs and consequences}\label{sec:proof}

\begin{lem}\label{lem:ring-square}
Let $G\subseteq\R$ be an additive subgroup closed under binary
interleaving. Then $\dimH(G\times G)=2\dimH G$.
\end{lem}

\begin{proof}
If $G$ is discrete, both dimensions are zero. Otherwise choose
$c\in G$ with $0<c<1$, and put $B=G\cap[0,1)$ and $H=G\cap[0,c)$.
Since
\[
G=\bigcup_{k\in\Z}(kc+H),
\]
countable stability and $H\subseteq B\subseteq G$ give
\[
\dimH G=\dimH B,
\qquad \dimH(G\times G)=\dimH(B\times B).
\]
Let $A=\beta^{-1}[B]$, including all binary codes of points in $B$.
For $X,Y\in A$, the closure assumption gives $\beta(X\oplus Y)\in G$.
Also $\beta(X\oplus Y)<1$, since $\beta(X),\beta(Y)<1$.
Thus $\beta(X\oplus Y)\in B$, so $J[A\times A]\subseteq A$.
Lemmas~\ref{lem:binary} and~\ref{lem:join} imply
\begin{align*}
\dimH(G\times G)
&=\dimH(A\times A)\\
&=2\dimH J[A\times A]\le2\dimH A=2\dimH G.
\end{align*}
The reverse inequality follows from \eqref{eq:product-lower}.
\end{proof}

\begin{proof}[Proof of Theorem~\ref{thm:main}]
Let $d=\dimH S$, and suppose that $0<d<1$. By
Lemma~\ref{lem:ring-square}, $E:=S\times S$ has Hausdorff dimension $2d$.

For $a,b\in S$, let $\ell_{a,b}=\{(x,y):y=ax+b\}$, and set
$\mathcal L=\{\ell_{a,b}:a,b\in S\}$. Point--line duality gives
$\dimH\mathcal L=\dimH(S\times S)=2d$. Since $S$ is a subring,
\[
E\cap\ell_{a,b}=\{(x,ax+b):x\in S\}.
\]
The map $x\mapsto(x,ax+b)$ is bi-Lipschitz onto its image, so every such
intersection has Hausdorff dimension $d$. Applying Theorem~\ref{thm:OS}
with $s=d$ and $t=2d$ yields
\[
\dimH E\ge2d+\varepsilon(d,2d)>2d,
\]
contradicting $\dimH E=2d$. Hence $d\in\{0,1\}$.
\end{proof}

\begin{proof}[Proof of Corollary~\ref{cor:turing}]
By the definition of \(F_{\mathcal I}\), we have
\[
\beta^{-1}[F_{\mathcal I}\cap[0,1]]=\mathcal I.
\]
Thus join closure also makes \(F_{\mathcal I}\) closed under binary interleaving. Theorem 1.1, Lemma 2.1, and countable stability under integer translates yield
\[
\dim_{\mathrm H}\mathcal I=\dim_{\mathrm H}F_{\mathcal I}\in\{0,1\}.
\]
\end{proof}

\begin{cor}\label{cor:models}
Suppose that $M\subseteq N$ are transitive models of ZFC and that
$\R^M$ is a set in $N$. Then
\[
\dimH^N(\R^M)\in\{0,1\}.
\]
\end{cor}

\begin{proof}
Work in $N$. The set
\[
\mathcal I:=\beta^{-1}[\R^M\cap[0,1]]=(2^\omega)^M
\]
is a Turing ideal, and $F_{\mathcal I}=\R^M$.
The result follows from Corollary~\ref{cor:turing}.
\end{proof}

\end{document}